\documentclass[letterpaper, 10 pt, conference]{ieeeconf}  % Comment this line out
\IEEEoverridecommandlockouts                              % This command is only
\usepackage{epsfig}
\usepackage{subfigure}
\usepackage{amssymb}
\usepackage{amstext}
\usepackage{amsmath}
\newtheorem{theorem}{\textbf{Theorem}}

\newtheorem{lemma}{\textbf{Lemma}}

\newtheorem{corr}{\textbf{Corollary}}

\title{\LARGE \bf
Optimal Distributed Controller Synthesis for Chain Structures: Applications to Vehicle Formations
}

\author{Omid Khorsand, Assad Alam and Ather Gattami% <-this % stops a space
\thanks{This work was supported in part by the Swedish Research Council.}% <-this % stops a space
\thanks{O. Khorsand and A. Gattami are with the Department of Electrical Engineering,
        KTH-Royal Institute of Technology, SE-100 44 Stockholm, Sweden
        {\tt\small khorsand,gattami@kth.se}}%
\thanks{A. Alam is with Research and Development, Scania CV AB, 151 87
        S\"{o}dert\"{a}lje, Sweden
        {\tt\small assad.alam@scania.com}}%
}

\begin{document}

\maketitle
\thispagestyle{empty}
\pagestyle{empty}

%%%%%%%%%%%%%%%%%%%%%%%%%%%%%%%%%%%%%%%%%%%%%%%%%%%%%%%%%%%%%%%%%%%%%%%%%%%%%%%%
\begin{abstract}

We consider optimal distributed controller synthesis for an interconnected system subject to communication constraints, in linear quadratic settings. Motivated by the problem of finite heavy duty vehicle platooning, we study systems composed of interconnected subsystems over a chain graph. By decomposing the system into orthogonal modes, the cost function can be separated into individual components. Thereby, derivation of the optimal controllers in state-space follows immediately. The optimal controllers are evaluated under the practical setting of heavy duty vehicle platooning with communication constraints. It is shown that the performance can be significantly improved by adding a few communication links. The results show that the proposed optimal distributed controller performs almost as well as the centralized linear quadratic Gaussian controller and outperforms a suboptimal controller in terms of control input. Furthermore, the control input energy can be reduced significantly with the proposed controller compared to the suboptimal controller, depending on the vehicle position in the platoon. Thus, the importance of considering preceding vehicles as well as the following vehicles in a platoon for fuel optimality is concluded.

\end{abstract}

%%%%%%%%%%%%%%%%%%%%%%%%%%%%%%%%%%%%%%%%%%%%%%%%%%%%%%%%%%%%%%%%%%%%%%%%%%%%%%%%
\section{INTRODUCTION}
\label{sec:introduction}
The systems to be controlled are, in many application
domains, getting larger and more complex. When there is interconnection
between different dynamical systems, conventional
optimal control algorithms provide a solution where
centralized state information is required. However, it is often
preferable and sometimes necessary to have a decentralized
controller structure, since in many practical problems, the
physical or communication constraints often impose a specific
interconnection structure. Hence, it is interesting to
design decentralized feedback controllers for systems of a
certain structure and examine their overall performance.

The control problem in this paper is motivated by systems, generally referred to as vehicle platooning, involving a chain of closely spaced heavy duty vehicles (HDVs). Information technology is paving its path into the transport industry, enabling the possibility of automated control strategies. Governing vehicle platoons by an automated control strategy, the overall traffic flow is expected to improve \cite{IoannouChien93} and the road capacity will increase significantly \cite{DeSchutter:99}. With radar sensors, each vehicle is able to measure the relative distance and velocity of the preceding vehicle. The radar measurements are conveyed further down the chain of vehicles through wireless communication. By traveling at a close intermediate spacing, the air drag is reduced for each vehicle in the platoon. Thereby, the control effort and inherently the fuel consumption can be reduced significantly. However, as the intermediate spacing is reduced the control becomes tighter due to safety aspects; mandating an increase in control action through additional acceleration and braking. Hence, it is of vast interest for the industry to find a fuel optimal control. Thus, with limited information and control input constraints, the control objective is to maintain a predefined headway to the vehicle ahead based upon local state measurements, which makes it a decentralized control problem.

Decentralized control problems are still intractable in general. One approach has been to classify specific information patterns leading to linear optimal controllers. In \cite{c4}, sufficient conditions are given under which optimal controllers are linear in the linear quadratic setting. An important result was given in \cite{c5} which showed that for a new information structure, referred to as \emph{partially nested}, the optimal policy is linear in the information set. In \cite{rantzer2006}, stochastic linear quadratic control problem was solved under the condition that all the subsystems have access to the global information from some time in the past. \cite{bamie}, showed that the constrained linear optimal decision problem for infinite horizon linear quadratic control, can be posed as an infinite dimensional convex optimization problem, given that the considered system is stable. Control for chain structures in the context of platoons has been studied through various perspectives, e.g., \cite{Bamieh08,Barooah05,RoggeAeyels,BamiehJovanovic05,SudinCook04,Hedrick96,Varaiya93}. It has been shown that control strategies may vary depending on the available information within the platoon. However, communication constraints have not in general been considered in control design for platooning applications.
%Furthermore, a duality result for the distributed estimation and control over graphs under partially nested information pattern has been shown and a solution for distributed estimation has been found in \cite{c81} and \cite{c82}.

The aim of this study is to synthesize controllers for a practical decentralized system composed of $M$ interacting systems over a chain. We minimize a quadratic cost under the partially nested information structure. This problem is known to have a linear optimal policy, \cite{c5} and \cite{c9}. However, most existing approaches do not provide explicit optimal controller formulae and, the order of the controllers can be large \cite{thesis}, which makes the implementation difficult. Some work has been focused on finding numerical algorithms to these problems, \cite{num1} and \cite{num2}. Recently, state-space solutions to the so-called two-player state-feedback $H_2$ version of this problem have been given in \cite{swigart10}. Also, in \cite{shah10}, using concepts from order theory, a control architecture has been proposed for systems having the structure of a partially ordered set. In contrast, we construct conditional estimates based on the information shared among the controllers. Thereby, we show how to decompose the states, control inputs, and as a result, the cost function into independent terms. Having the cost function decomposed into individual pieces, analytical derivation of the optimal controllers follows immediately.

The main contribution of this paper is to introduce a simple decomposition scheme to construct optimal decentralized controllers with low computational complexity for chain structures which is applicable to intelligent transportation systems in terms of automated platooning. Derived from the characteristics of actual Scania HDV's, we present a discrete system model that includes physical coupling with a preceding vehicle. In the context of HDV platooning, we explicitly study systems composed of two and three interconnected subsystems over a chain structure. The proposed control scheme accounts for a constrained communication pattern among the vehicles and hence reduces the communications compared to a centralized information pattern where full state information is available to each controller. We also evaluate the performance of the optimal controllers for a typical scenario in HDV platooning under normal operating conditions, with respect to the imposed information constraints.

The outline of the remainder of this paper is as follows. First we specify the problem that we are considering in Section~\ref{sec2}. Then, the finite and infinite horizon optimal controller formulation for the simplest case, the two-vehicle problem, will be presented in Section~\ref{chapter2p}. In Section~\ref{chapter3p}, we will show how the decomposition scheme can be extended to the case of three interconnected subsystems. We apply the three-vehicle optimal distributed controller to the example of HDV platooning in Section~\ref{sec:Simulations} where we evaluate the proposed controller in comparison with the optimal centralized controller and a suboptimal decentralized controller.

%%%%%%%%%%%%%%%%%%%%%%%%%%%%%%%%%%%%%%%%%%%%%%%%%%%%%%%%%%%%%%%%%%%%%%%%%%%%%%%%
\textbf{Notation}.
We denote a matrix partitioned into blocks by $A=[A_{ij}]$, where $A_{ij}$ denotes the block matrix of $A$ in block position $(i,j).$ The submatrix of $A$ formed by row partitions $i$ through $j$ and column partitions $k$ through $l$ will be denoted by $A[i:j,k:l]$:
$$A{[i:j,k:l]}=\begin{bmatrix}A_{ik} &A_{i(k+1)} &\cdots &A_{il} \\A_{(i+1)k} &A_{(i+1)(k+1)} &\cdots &A_{(i+1)l} \\\vdots &\vdots &\ldots &\vdots \\A_{jk} &A_{j(k+1)} &\cdots &A_{jl} \end{bmatrix}.$$
The expected value of a random variable $x$ is denoted by $\mathbf{E}\{x\}$. The conditional expectation of $x$ given $y$ is denoted by $\mathbf{E}\{x|y\}$. The trace of a matrix $A$ is denoted by $\mathbf{Tr}\{A\}$, and the sequence $x(0),~ x(1),...,~x(t)$, is denoted by $x(0:t)$.
\section{\textsc{System Model and Problem Statement}}
\label{sec2}
In this section we present the physical properties of the
system that we are considering. We state the nonlinear dynamics
of a single vehicle and the model for the aerodynamics,
which induces the physical coupling. Then we present
the linear discrete system model for a heterogeneous HDV
platoon and its associated cost function. Finally, the problem formulation
is given.
\subsection{System Model}
\begin{figure}[t]
\begin{center}
\includegraphics[width=7.4cm]{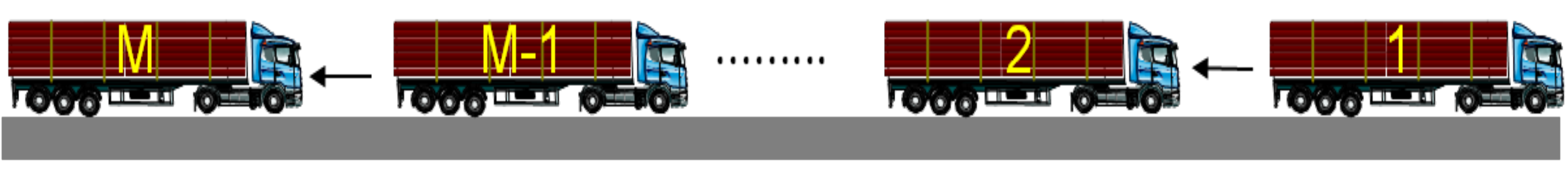}    % The printed column width is 8.4 cm.
\caption{The figure shows a platoon of $M$ heavy duty vehicles, where each vehicle is only able to communicate with the preceding vehicles.}
\label{fig:platoon}
\end{center}
\end{figure}
We consider an HDV platoon as depicted in Figure~\ref{fig:platoon}. The state equation of a single HDV is modeled as \cite{Sahlholm11},

\begin{equation}
	\begin{aligned}
	%\begin{split}
 		\dot{s} 		&= v,\\
		m_t\dot{v}	&=F_{engine} -F_{brake}-F_{air drag}(v)\\
 									& \hspace{15.6mm}				 -F_{roll}(\alpha)-F_{gravity}(\alpha), \\
 									&=k_uu -k_bF_{brake}-k_dv^2 \\
 									& \hspace{11.6mm}-k_{fr}\cos\alpha-k_g\sin\alpha,
		%\end{split}
	\end{aligned}
\label{eq:Forces}
\end{equation}

\noindent where $v$ is the vehicle velocity, $m_t$ denotes the accelerated mass and $u\in \mathbb{R}$ denotes the net engine torque. $k_u, k_b, k_d, k_{fr}$, and $k_g$ denote the characteristic vehicle and environment coefficients for the engine, brake, air drag, road friction, and gravitation respectively.

The aerodynamic drag has a strong impact on an HDV, since it can amount up to 50\,\% of the total resistive forces at full speed. When traveling at short intermediate spacings, the wind resistance is reduced significantly. Hence, a physical coupling is induced between each vehicle in a platoon. To account for the aerodynamics the air drag characteristic coefficient in \eqref{eq:Forces} can be modeled as

\begin{equation*}
\tilde{k}_d=k_d(1-\frac{\Phi(d)}{100}),
\end{equation*}

\noindent where $\Phi(d)=\kappa_1d + \kappa_2$, $0 \leq d \leq 65$ is the longitudinal relative distance between two vehicles, and $\kappa_1, \kappa_2$ are adjusted according to the graphical model given in Figure~\ref{fig:airDrag}.

\begin{figure}[t]
\begin{center}
\includegraphics[width=7.4cm]{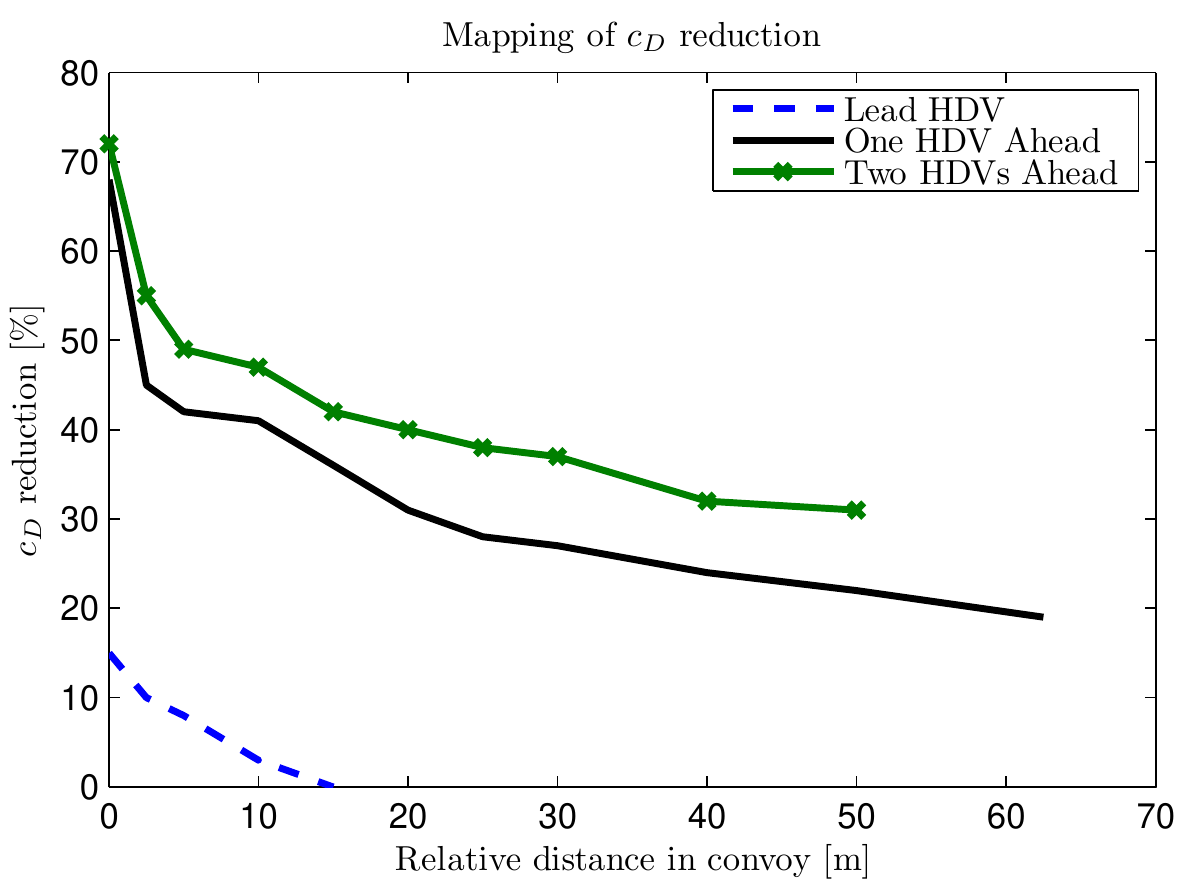}    % The printed column width is 8.4 cm.
\caption{The empirical air drag coefficient $c_D$ as a function of the intermediate spacing $d$. Adapted from \cite{Hucho}. Similar findings are found in \cite{DaimlerChrysler}.}
\label{fig:airDrag}
\end{center}
\end{figure}

The velocities do not deviate significantly for the vehicles with respect to the lead vehicle's velocity in an automated HDV platoon. Thus, a linearized model should give a sufficient description of the system behavior. By linearizing and applying a one step forward discretization to \eqref{eq:Forces}, the discrete model with respect to a set reference velocity, an engine torque which maintains the velocity, a fixed spacing between the vehicles, and a constant slope is hence given by

\begin{equation} \label{system}
x(t+1) = Ax(t) + Bu(t)+w(t),
\end{equation}

\noindent where
{\small{
\begin{align*}
\begin{split}
A&= \begin{bmatrix}{\Theta_1} & 0 & 0 & 0 & 0 & \cdots & 0 & 0 & 0 \\
1 & 1 & -1 & 0 & 0 & \cdots & 0 & 0 & 0 \\
0 & \delta_2 & \Theta_2 & 0 & 0 & \cdots & 0 & 0 & 0 \\
0 & 0 & 1 & 1 & -1 & \cdots & 0 & 0 & 0 \\
0 & 0 & 0 & \delta_3 & \Theta_3 & \cdots & 0 & 0 & 0 \\
\vdots & \vdots & \vdots & \vdots & \vdots & \ddots & \vdots & \vdots & \vdots \\
0 & 0 & 0 & 0 & 0 & \cdots & \Theta_{M-1} & 0 & 0 \\
0 & 0 & 0 & 0 & 0 & \cdots & 1 & 1 & -1 \\
0 & 0 & 0 & 0 & 0 & \cdots & 0 & \delta_M & \Theta_M \\
\end{bmatrix},
\end{split}
\end{align*}
\begin{align}
\begin{split}
B&= \begin{bmatrix} k_{u_1} & 0 & 0 & \cdots & 0 \\
0 & 0 & 0 & \cdots & 0 \\
0 & k_{u_2} & 0 & \cdots & 0 \\
0 & 0 & 0 & \cdots & 0 \\
0 & 0 & k_{u_3} & \cdots & 0 \\
\vdots & \vdots & \vdots & \ddots & \vdots \\
0 & 0 & 0 & \cdots & 0 \\
0 & 0 & 0 & \cdots & k_{u_M} \end{bmatrix}, \quad x=\begin{bmatrix}  {v_1} \\
{d_{12}}  \\
{v_2} \\
{d_{23}} \\
{v_3} \\
\vdots \\
{v_{M-1}} \\
{d_{(M-1)M}} \\
{v_M} \end{bmatrix}, \\
u&=\begin{bmatrix}  {u_1} \\
{u_2}  \\
{u_3} \\
\vdots \\
{u_M} \end{bmatrix}, \begin{array}{rl}
\Theta_1 & =T_s(1-2k_dv_0),\\
\Theta_i & =-T_s2k_d\Phi{(d_0)}v_0, \quad i=2,\dots, M,\\
\delta_i & =-T_s\kappa_1k_dv_0^2,
\end{array}
\end{split}
\label{eq:plantModel}
\end{align}
}}
\noindent where $\delta_i$ denotes the physical coupling with a preceding vehicle and $T_s$ is the sampling time. The derived HDV platoon model in \eqref{eq:plantModel} has a lower block triangular structure, which can generally be stated as

\begin{align}
\hspace{-0.2cm}\begin{bmatrix} x_{1}(t+1) \\
x_{2}(t+1) \\
x_{3}(t+1) \\
\vdots \\
x_{M}(t+1)
 \end {bmatrix}=\begin {bmatrix}A_{11} &  0 & 0 & \cdots & 0\\
 A_{21} & A_{22} & 0 & \cdots & 0\\
 0 & A_{32} & A_{33}  & \cdots & 0\\
\vdots & \vdots & \vdots & \ddots & \vdots \\
0 & 0 & 0  & \cdots & A_{MM}
\end {bmatrix}&\begin{bmatrix} x_{1}(t) \\
 x_{2}(t) \\
 x_{3}(t) \\
\vdots \\
x_{M}(t)
 \end {bmatrix}\nonumber\\
+\begin{bmatrix} B_{1} &  0 &  0 & \cdots & 0\\
 0 & B_{2} &  0 & \cdots & 0\\
 0 &  0 &  B_{3} & \cdots & 0\\
 \vdots &  \vdots &  \vdots & \ddots & \vdots\\
 0 &  0 &  0 & \cdots & B_{M}\\
 \end {bmatrix}
\begin{bmatrix} u_{1}(t) \\
 u_{2}(t) \\
 u_{3}(t)\\
\vdots \\
u_{M}(t)
 \end {bmatrix}+&\begin{bmatrix} w_1(t) \\ w_2(t) \\ w_3(t)\\ \vdots \\ w_M(t)\end{bmatrix}
\label{eq:GeneralPlantModel}
\end{align}

\noindent where the corresponding vehicle states for each subsystem are
\begin{equation*}
x_1(t)=v_1(t),\quad x_i(t)=\begin{bmatrix}d_{i-1,i} \\ v_i
\end{bmatrix},\quad i=2,\dots,M.
\end{equation*}

\subsection{Performance Criteria}
The performance criteria of an HDV platoon can be mapped into quadratic costs. Hence, we formulate the weight parameters for a quadratic cost function based upon performance and safety objectives. The objective of the lead vehicle is to minimize the fuel consumption and control input, while maintaining a set reference velocity. The objective of the follower vehicles in addition, is to follow the preceding vehicles velocity, while maintaining a set intermediate spacing. The intermediate spacing reference could be constant or, as in this case, time varying. It is determined by setting a desired time gap $\tau$\,s, which in turn determines the spacing policy as

\begin{equation*}
d_{ref}(t)={\tau}v(t).
\end{equation*}
\noindent Thereby, the vehicles will maintain a larger intermediate spacing at higher velocities. Hence, the weights for an $M$ HDV platoon can be set up as
{\small{
\begin{align}
\begin{split}
J(u^*)=&\min_{u}~\sum_{t=0}^{N-1}\Big(\sum_{i=2}^{M}w_i^{\tau}(d_{(i-1)i}(t)-\tau v_i(t))^2 \\
&\hspace{13.5mm}+ w_i^{\Delta{v}}(v_{i-1}(t)-v_i(t))^2\\
& \hspace{13.5mm}+w_i^{d}d_{(i-1)i}^2(t)+\sum_{i=1}^{M}w_i^{v}v_{i}^2(t)+w_i^{u_i}u_i^2(t)\Big) \\
=&\min_{u}~\sum_{t=0}^{N-1}\sum_{i=2}^{M}\begin{bmatrix} v_{i-1}(t)\\
d_{(i-1)i}(t)\\
v_i(t)\end{bmatrix}^TQ_i\begin{bmatrix} v_{i-1}(t)\\
d_{(i-1)i}(t)\\
v_i(t)\end{bmatrix}+R_iu_i^2(t)\\
&\hspace{13.5mm}+w_1^{v}v_{1}^2(t)+w_1^{u_1}u_1^2(t)
\end{split}
\label{eq:weights}
\end{align}
}}
\noindent where
{\small{
\begin{align}
\begin{split}
&Q_i=\begin{bmatrix}
w_i^{\Delta{v}} & 0 & -w_i^{\Delta{v}} \\
0 & w_i^{d}+w_i^{\tau} & -\tau{w}_i^{\tau} \\
-w_i^{\Delta{v}} & -\tau{w}_i^{\tau} & \tau^2w_i^{\tau}+w_i^{\Delta{v}}+w_i^{v}
\end{bmatrix},\\
&R_i=w_i^{u_i}.
\end{split}
\label{eq:weightCosts}
\end{align}
}}

The weights in \eqref{eq:weights} give a direct interpretation of how to enforce the objectives for a vehicle traveling in a platoon. The value of $w_i^{\tau}$ determines the importance of not deviating from the desired time gap. Hence, a large $w_i^{\tau}$ puts emphasis on safety. $w_i^{\Delta{v}}$ creates a cost for deviating from the velocity of the preceding vehicle, and $w_i^{u_{i}}$ punishes the control effort which is proportional to the fuel consumption. The following terms, $w_i^{d}, w_i^{v}$, put a cost on the deviation from the linearized states. Note that the main objective is to maintain a set intermediate distance, while maintaining a fuel efficient behavior. Therefore, $w_i^{\tau}, w_i^{\Delta{v}}$ and $w_i^{u_{i}}$ must be set larger than the remaining weights. The weights are chosen such that $Q$ is positive semidefinite and $R$ is positive definite.
\subsection{Problem Formulation}
%Consider a discrete time linear system
%\begin{align} \label{system} x(t+1)&=Ax(t)+Bu(t)+w(t),
%\end{align}
%composed of $M$ interconnected subsystems, where the system matrices are partitioned into blocks as $A=[A_{ij}],~B=[B_{ij}],~i,j=1,...,M$, and $x$, $u$, and $w$ are the overall state, control input and disturbance vectors. These vectos are partitioned as
%\begin{align*}  x(t)&=\begin{bmatrix}x^T_{1}(t)&x^T_2(t)& \cdots &x^T_M(t) \end{bmatrix}^T, \\
%                u(t)&=\begin{bmatrix}u^T_{1}(t)&u^T_2(t)& \cdots &u^T_M(t) \end{bmatrix}^T, \\
%                w(t)&=\begin{bmatrix}w^T_{1}(t)&w^T_2(t)& \cdots &w^T_M(t) \end{bmatrix}^T.
%\end{align*}
%The interconnection structure can be described as follows. If the state $x_j(t)$  affects the state $x_i(t + 1)$, then $A_{ij} \neq 0$, otherwise $A_{ij} = 0$. This structure can be represented by a directed graph of order $M$. The graph has an arrow from node $j$ to $i$ if and only if $A_{ij}\neq 0$.

Although the approach used in this paper is applicable for systems over general acyclic graphs, for simplicity we will concentrate on two simple chain structures, which we refer to as two- and three-vehicle chains.
The aim is to synthesize controllers under imposed communication constraints.

For the two-vehicle chain the system matrices have the sparsity structure as
\begin{align} \label{2p} A=\begin{bmatrix}A_{11} & 0\\A_{21}&A_{22} \end{bmatrix},\hspace{2mm} B=\begin{bmatrix}B_{1} & 0\\0&B_{2} \end{bmatrix}. \end{align}
Assume $\{w(t)\}$ is a sequence of mutually independent Gaussian vectors with zero mean values and covariance given by
\begin{align*}
&\mathbf{E}\{w(k)w^T(l)\}=\begin{bmatrix}W_{1}&0\\0&W_{2}\end{bmatrix}\delta(k-l).
\end{align*}
It is assumed that $x(0)=0.$

In this system, the dynamics of subsystem 1 (Vehicle 1) propagates to subsystem 2 (Vehicle 2) but not vice-versa. If both subsystems have access to the global state measurements the information structure would be classical, and the optimal linear controller could be obtained from the linear quadratic control theory. However, in the practical setting of HDV platooning the lead vehicle only has its own state information, whereas the follower vehicle can also measure the states of the preceding vehicle through radar sensors. Therefore, we consider the case in which $u_2$ has access to the overall measurement history, while $u_1$ has access to its own measurements. Let $\mathbb{I} ^t_i$ denote the information set of controller $i$ at time $t$. Then
\begin{align}\label{constraints}\mathbb{I}^t_1=\{x_1(0:t)\},\hspace{2mm} \mathbb{I}^t_2=\{x(0:t)\}.\end{align}
This information pattern is not classical anymore and is a simple case of a \emph{partially nested} information structure. This is one of a few non-classical information patterns for which the optimal policy is known to be unique and linear in the information set.
For the chain of three vehicles, the matrices are given by
\begin{equation}
A=\begin{bmatrix}A_{11} & 0 & 0\\A_{21}&A_{22}&0\\0&A_{32}&A_{33} \end{bmatrix}, B=\begin{bmatrix}B_{1} & 0 & 0\\0&B_{2}&0\\0&0&B_3 \end{bmatrix}.
\label{eq:AMatrixBlocks}
\end{equation}
\noindent Here, $\{w(t)\}$ is a Gaussian disturbance vector with covariance given by
\begin{align*}
\mathbf{E}\{w(k)w^T(l)\}=\begin{bmatrix}W_{1}&0&0\\0&W_{2}&0\\0&0&W_{3}\end{bmatrix}\delta(k-l).
\end{align*}
To maintain partially nestedness, the information set for the controllers is given by
\begin{equation}\label{infoset}
\begin{array}{rcl}
\mathbb{I}^t_1=\{x_1(0:t)\}, ~~\mathbb{I}^t_2=\{x_1(0:t), x_2(0:t)\},  \\
\mathbb{I}^t_3=\{x_1(0:t), x_2(0:t), x_3(0:t)\}.~~~~~~~~~~~~~~
\end{array}
\end{equation}
where only one communication link is needed from vehicle 1 to vehicle 3, since vehicle 2 and 3 can measure the preceding vehicle states with on-board radar sensors.

Thus, the problem that we solve is finding an analytical formulation for optimal controllers constrained to specified information sets that minimize the infinite-horizon quadratic cost
\begin{equation} \label{cost}
\lim_{N \rightarrow \infty} \frac{1}{N} \mathbf{E} \sum_{t=0}^{N-1} (x^T(t)Qx(t)+u^T(t)Ru(t)),
\end{equation}
subject to the given system dynamics and performance objectives. We first give an explicit solution for the two-vehicle problem defined by \eqref{2p} and \eqref{constraints}, where the intuition behind the solution is derived. To show how the proposed technique can be applied to more general chains, we then present an explicit solution for the three-vehicle problem with dynamics given in~\eqref{eq:AMatrixBlocks} subject to constraints in~\eqref{infoset}.

%Then the extension to three vehicles, defined by \eqref{eq:AMatrixBlocks} and %\eqref{infoset}, is given because adding more vehicles eventually reduces to %the two vehicle problem by decomposition.
%\vspace{0.5 mm}
\section{TWO-VEHICLE CHAIN} \label{chapter2p}
The aim of this section is to present the optimal control synthesis for the simplest case of the problem which is a chain of two vehicles. The derivation given in this section explains the decomposition idea and the structure of the controllers. First, we shall present the optimal controller in Section \ref{main1}. Next, the derivation of the time-varying and the stationary controller will be explained in Section \ref{derivation}. Finally, we conclude with some remarks in Section \ref{remarks1}.
\subsection{\textbf{Main Result}} \label{main1}
\begin{theorem} \label{theorem1}
Assume that
\renewcommand{\theenumi}{\roman{enumi}}
\begin{enumerate}
                \item $(A,B)$ is stabilizable,
                \item $(A_{22},B_{2})$ is stabilizable,
                \item $( {Q}, A)$ is detectable,
                \item $( {Q_{22}}, A_{22})$ is detectable.
\end{enumerate}
Then, the optimal controller for the two-vehicle chain is given by:
\begin{align*}
\eta(t+1)&=(A_{22}-B_2L_{22})\eta(t)+[A_{21}-B_2L_{21}~~ 0]\begin{bmatrix}x_1(t)\\x_2(t)\end{bmatrix} \\
u(t)&=-\begin{bmatrix}L_{12}\\L_{22}-L^2\end{bmatrix}\eta(t)-\begin{bmatrix}L_{11} & 0\\L_{21} &L^2\end{bmatrix}\begin{bmatrix}x_1(t)\\x_2(t)\end{bmatrix}.
\end{align*}
and the optimal cost is
$$\mathbf{Tr}(X_{11}W_1)+\mathbf{Tr}(YW_2).$$
The matrices $X$ and $Y$ are the positive semidefinite stabilizing solutions to the Riccati equations
\begin{align*} X&=A^TXA+Q-A^TXB(B^TXB+R)^{-1}B^TXA, \\
               Y&=A_{22}^TYA_{22}+Q_{22}-A_{22}^TYB_2(B_2^TYB_2+R_{22})^{-1}B^T_{2}YA_{22},
\end{align*}
and the matrix $X$ is partitioned into blocks compatible with the partitions of $A$:
$$X=[X_{ij}],~~i,j=1, 2.$$
The gain matrices $L^1$ and $L^2$ are given by
\begin{align*}
L^1&=(R+B^TXB)^{-1}B^TXA, \\
L^2&=(R_{22}+B_2^TYB_2)^{-1}B_2^TYA_{22},
\end{align*}
and $L^1$ is partitioned into blocks according to
$$L^1=[L_{ij}], ~~i,j=1,2.$$
\
\end{theorem}
%%%%%%%%%%%%%%%%%%%%%%%%%%%%%%%%%%   %  PROOF   %      %%%%%%%%%%%%%%%%%%%%%%%%%%%%%%%%%%%%%
\

Before giving the proof of the theorem, we need to state the following lemma and corollary.
\begin{lemma}  \label{lemma1}
Consider the system described by (\ref{system}), we introduce the following Riccati equation
\begin{align*}P(t)=&A^TP(t+1)A+Q-A^TP(t+1)B\times \\
                  &(B^TP(t+1)B+R)^{-1}B^TP(t+1)A,\end{align*}
for $t=0,\hdots,N$, with the end condition $P(N) = Q$, where $Q$ is positive semidefinite. Then, \begin{align*}&x^T(N)Qx(N)+\sum_{t=0}^{N-1} (x^T(t)Qx(t)+u^T(t)Ru(t))=   \\
&x^T(0)P(0)x(0)+\sum_{t=0}^{N-1}(u(t)+L(t)x(t))^T  \\
              &\times(B^TP(t+1)B+R)(u(t)+L(t)x(t)) \\
              &+\sum_{t=0}^{N-1}2(w^T(t)P(t+1)(Ax(t)+Bu(t)))\\
              &+\sum_{t=0}^{N-1} w^T(t)P(t+1)w(t)
\end{align*}
where $L(t)$ is given by
$$L(t)=(R+B^TP(t+1)B)^{-1}B^TP(t+1)A.$$
\end{lemma}
\begin{proof}
See for example \cite{astrom-wittenmark}. \end{proof}
\begin{corr} \label{cor1}
Assume that $x(0)=0$, $\{w(t)\}$ is a sequence of uncorrelated Gaussian variables with the covariance $W$, and $w(t)$ is independent of $x(t)$ and $u(t)$. Then,
\begin{align*}&\mathbf{E}\{x^T(N)Qx(N)+\sum_{t=0}^{N-1} (x^T(t)Qx(t)+u^T(t)Ru(t))\}=   \\
              &\mathbf{E}\{ \sum_{t=0}^{N-1}(u(t)+L(t)x(t))^T(B^TP(t+1)B+R)\times  \\
              &(u(t)+L(t)x(t))\}+\sum_{t=0}^{N-1}\mathbf{Tr}(P(t+1)W).\end{align*}
where $L(t)$ and $P(t)$ are given in Lemma \ref{lemma1}.
\end{corr}

\subsection{\textbf{Optimal Controller Derivation}} \label{derivation}
Based on the information constraints in (\ref{constraints}), we want to find the controllers restricted to the following structure:
\begin{equation}\label{f12}
\begin{array}{rcl}
u_1(t)&=&f_{1}(x_1(0:t)), \\
u_2(t)&=&f_{2}(x(0:t)),
\end{array}
\end{equation}
where $f_{i}$, $i=1,2$, denote linear functions in their arguments.

To derive the optimal controller, we will first consider a finite-horizon version of the problem with the cost function given by
$$J=\mathbf{E}~x^T(N)Qx(N)+\mathbf{E}\sum_{t=0}^{N-1} (x^T(t)Qx(t)+u^T(t)Ru(t)).$$
To find a structure for the controllers, we decompose the state variable into two independent terms as
$$x(t)=z^1(t)+z^2(t),$$
where $z^1(t):=\mathbf{E}\{x(t)|x_1(0:t)\}$, and $z^2(t):=x(t)-z^1(t)$. The term $z^1$ is the conditional estimate of $x$ given the information shared between the controllers, namely $x_1(0:t)$, and $z^2$ is the estimation error. Let these vectors be partitioned as $z^i(t)=\begin{bmatrix}z^i_1(t)\\z^i_2(t)\end{bmatrix},i=1,2$. Clearly, the first component of $z^1(t)$ is $x_1(t)$. Hence
$$z^1(t)=\begin{bmatrix}x_1(t)\\z^1_2(t)\end{bmatrix}, ~~z^2(t)=\begin{bmatrix}0\\z^2_2(t)\end{bmatrix}.$$

Analogously, the control input is decomposed as $u(t)=u^1(t)+u^2(t)$, where $u^1$ and $u^2$ are independent terms defined by
$$u^1(t):=\mathbf{E}\{u(t)|x_1(0:t)\}, ~~u^2(t):=u(t)-u^1(t).$$

\begin{lemma} \label{lemma2}
The update equations for $z^1$ and $z^2$ are given by:
\begin{align*}
z^1(t+1)&=Az^1(t)+Bu^1(t)+\begin{bmatrix} w_1(t)\\ 0 \end{bmatrix}  \\
z^2(t+1)&=Az^2(t)+Bu^2(t)+\begin{bmatrix} 0 \\ w_2(t) \end{bmatrix}
\end{align*}
\end{lemma}
\begin{proof} See Appendix. \end{proof}
Now, considering $u(t)$ on the form given by (\ref{f12}) we find that
\begin{align*}
u^1(t)&=\mathbf{E}\{u(t)|x_1(0:t)\}\\
          &=\mathbf{E} \left\{ \begin{bmatrix}f_{1}(x_1(0:t))\\f_{2}(x(0:t))\end{bmatrix}\bigg |x_1(0:t) \right\}\\
          &=\begin{bmatrix}f_{1}(x_1(0:t))\\f_{2}(z^1(0:t)) \end{bmatrix},
\end{align*}
where the last equality follows from the fact that
$\mathbf{E}\{f_{2}(x(0:t))|x_1(0:t)\}=f_{2}(\mathbf{E}\{x(0:t)|x_1(0:t)=f_{2}(z^1(0:t))\})$. Thus, $u^2$ has the structure
$$u^2(t)=\begin{bmatrix}0\\f_{2}(z^2(0:t)) \end{bmatrix}.$$
By partitioning these vectors as $u^i=\begin{bmatrix}u^i_1 \\ u^i_2\end{bmatrix}~i=1,2$, it can be seen that $u_1^2(t)=0$, so the control input for subsystem 1 is given as the first component of the vector $u^1$, while subsystem 2's input is separated into the two independent terms, namely $u^1_2$, and $u^2_2$. In other words, we have
$$\begin{bmatrix}u_1(t)\\u_2(t)\end{bmatrix}=\underbrace{\begin{bmatrix}u_1(t)\\u^1_2(t)\end{bmatrix}}_{u^1(t)}+\underbrace{\begin{bmatrix}0\\u^2_2(t)\end{bmatrix}}_{u^2(t)}.$$
Decomposition of the states and inputs into independent terms and having $u^1$ and $u^2$ given as functions of $z^1$ and $z^2$ (which are independent terms) implies that the vectors $\begin{bmatrix}z^1(t)\\u^1(t)\end{bmatrix}$ and $\begin{bmatrix}z^2(t)\\u^2(t)\end{bmatrix}$ are independent. As a result, $J$ can be decomposed as:
\begin{small}
\begin{align*}
&\underbrace{\mathbf{E}(z^1(N))^TQz^1(N)+\hspace{-1mm}\sum_{t=0}^{N-1}(z^1(t))^TQz^1(t)+(u^1(t))^TRu^1(t)}_{J_1} \\
&\hspace{-1mm}+\underbrace{\mathbf{E}(z^2(N))^TQz^2(N)+\hspace{-1mm}\sum_{t=0}^{N-1}\hspace{-1mm}(z^2(t))^TQz^2(t)+\hspace{-1mm}(u^2(t))^TRu^2(t)}_{J_2}
\end{align*}
\end{small}
Note that having $z^2_1$ and $u^2_1$ equal to zero implies that only the second component of $z^2$ is nonzero. The dynamics for this component can be written as
$$z^2_2(t+1)=A_{22}z^2_2(t)+B_2u^2_2(t)+w_2(t).$$
Noting that $w_i(t)$ is independent of $z^i(t)$, $u^i(t)$, we can apply Corollary \ref{cor1} to transform $J_1$ and $J_2$: $J=$
\begin{align}\label{quadratic1}&\mathbf{E}\hspace{-1mm}\sum_{t=0}^{N-1}(u^1(t)+L^1z^1(t))^T(B^TX(t+1)B+R) \times \nonumber \\
              &(u^1(t)+L^1z^1(t))+\mathbf{E}\hspace{-1mm}\sum_{t=0}^{N-1}(u^2_2(t)+L^2z^2_2(t))^T \times \nonumber \\
              &(B_2^TY(t+1)B_2+R_{22})(u^2_2(t)+L^2z^2_2(t))+                                                \nonumber \\
              &\sum_{t=0}^{N-1}(\mathbf{Tr}(X(t+1)\begin{bmatrix}W_1&0\\0&0\end{bmatrix})+\mathbf{Tr}(Y(t+1)W_2)),
\end{align}
where we also used $x(0)=0$.
The matrices $X(t)$ and $Y(t)$ are computed recursively by
\begin{align*} X(t)=&A^TX(t+1)A+Q-A^TX(t+1)B\times \\
                   & (B^TX(t+1)B+R)^{-1}B^TX(t+1)A, \\
                        Y(t)=&A_{22}^TY(t+1)A_{22}+Q_{22}-A_{22}^TY(t+1)B_2\times\\
               &(B_2^TY(t+1)B_2+R_{22})^{-1}B^T_{2}Y(t+1)A_{22},
\end{align*}
with the end conditions $X(N)=Q$, $Y(N)=Q_{22}$. The gain matrices $L^1$ and $L^2$ are given by
\begin{align*}
L^1(t)&=(R+B^TX(t+1)B)^{-1}B^TX(t+1)A, \\
L^2(t)&=(R_{22}+B_2^TY(t+1)B_2)^{-1}B_2^TY(t+1)A_{22}.
\end{align*}
Quadratic minimization of (\ref{quadratic1}) simply gives the optimal inputs $u^{1*}$, and $u_2^{2*}$ as
$$u^{1*}(t)=-L^1(t)z^1(t), ~~u_2^{2*}(t)=-L^2(t)z_2^2(t).$$
Let $X$ partitioned into appropriately sized blocks, $[X_{ij}],~~i,j=1,2$, then the optimal cost becomes,
$$J^*=\sum_{t=0}^{N-1}(\mathbf{Tr}(X_{11}(t+1)W_1)+\mathbf{Tr}(Y(t+1)W_2))$$
To find a mapping from $x$ to $u$, let $L^1$ partitioned into the blocks $[L_{ij}]$ so we get the control action $u^{1}$ on the form
\begin{equation*}\label{12}
    \begin{split}
      &{u_1}^*(t)=-L_{11}x_1(t)-L_{12}z^1_2(t), \\
      &{u^1_2}^*(t)=-L_{21}x_1(t)-L_{22}z^1_2(t),
    \end{split}
\end{equation*}
and the update equation for $z^1_2$ becomes
\begin{align*}\label{z12}
    z^1_2(t+1)&=(A_{22}-B_2L_{22})z^1_2(t)+(A_{21}-B_2L_{21})x_1(t).
\end{align*}
Finally, noting that $z^2_2(t)$ is given by $x_2(t)-z^1_2(t)$, the optimal controller can be rewritten on the form
$$\begin{bmatrix}u_1^*(t) \\ u_2^*(t)\end{bmatrix}=-L^1(t)\begin{bmatrix}x_1(t) \\ z_2^1(t)\end{bmatrix}-\begin{bmatrix} 0\\L^2(t) \end{bmatrix}(x_2(t)-z^1_2(t)).$$
Having derived the time-varying representation for the controllers, we now let $N$ go to infinity and obtain the steady-state form of the controller. Given the pairs $(A,B)$ and $(A_{22},B_{2})$ are stabilizable, and the pairs $( {Q},A)$ and $( {Q_{22}},A_{22})$, are detectable, $X(t)$ and $Y(t)$ converge the unique stabilizing solution to corresponding Riccati equations and as a result, $L^1(t)$ and $L^2(t)$ will tend to the steady-state values $L^1$ and $L^2$ given in Theorem 1. This will yield the controller representation given in the Theorem.

Finally, the optimal cost is computed as
\begin{align*} &\lim_{N \rightarrow \infty} \frac{1}{N}\sum_{t=0}^{N-1}(\mathbf{Tr}(X_{11}(t+1)W_1)+\mathbf{Tr}(Y(t+1)W_2)) \\
&=\mathbf{Tr}(X_{11}W_1)+\mathbf{Tr}(YW_2).
\end{align*}
\subsection{\textbf{Discussion}} \label{remarks1}
The state vector, $\begin{bmatrix}x_1(t)\\x_2(t)\end{bmatrix}$, is fed into the controller by a lower-triangular gain matrix, and hence $u_1$ is not dependent on $x_2$.

Note that $z_2^1(t)$ (same variable as $\eta(t)$ in Theorem 1) is the minimum-mean square estimate of $x_2(t)$ based on $x_1$ that is $\mathbf{E}\{x_2(t)|x_1(0),...,x_1(t)\}$. Therefore, $z_2^2(t)$ represents the error of this estimation.

For convenience, let $\hat{x}_{2|1}(t)$ denote the estimate of $x_2(t)$ based on history of $x_1$ and let $e_{2|1}(t)$ represent the estimation error, then we can write the controllers on a more intuitive form:
\begin{align*}
u_1^*(t)&=-L_{11}x_1(t)-L_{12}\hat{x}_{2|1}(t),                               \\
u_2^*(t)&=-L_{21}x_1(t)-L_{22}\hat{x}_{2|1}(t)-L^2e_{2|1}(t).
\end{align*}
Thus, both controllers use $\hat{x}_{2|1}$ instead of $x_2$ in the form of an optimal centralized control, however, controller 2 contains an additional term which is constructed based on the estimation error $e_{2|1}$.

We see that the order of each controller is equal to the state dimension of subsystem 2. It is easy to see that in a centralized information pattern where the value of $x_2$ is known to controller 1, the error term disappears and the controller reduces to a static gain similar to a classical linear quadratic regulator problem.

%We can now write the closed-loop dynamics of the system on the form:
%\begin{equation*}\label{closed-loop}
%    \begin{bmatrix}
%    x_1(t+1) \\
%    x_2(t+1) \\
%    \hat{x}_{2|1}(t+1)
%    \end{bmatrix}
%    =
%    A_{c} \begin{bmatrix}
%  x_1(t)  \\
%  x_2(t)  \\
%  \hat{x}_{2|1}(t)
%  \end{bmatrix}+
%  \begin{bmatrix}
%  I & 0 \\ 0 & I \\0 & 0 \end{bmatrix}
%  \begin{bmatrix}
%  w_1(t)  \\
%  w_2(t)
%  \end{bmatrix},
%\end{equation*}
%where $A_c=$
%\begin{align*}
%\begin{bmatrix}
%  A_{11} - B_{1}L_{11}               &         0                  & -B_{1}L_{12}                       \\
%  A_{21}  - B_{2}L_{21}  & A_{22} - B_{2}L_2           & - B_{2}(L_{22} - L_2)    \\
%  A_{21}  - B_{2}L_{21}  &         0                  & A_{22}  - B_{2}L_{22}
%\end{bmatrix} .\end{align*}

\section{THREE-VEHICLE CHAIN} \label{chapter3p}
The optimal controller synthesis for the three-vehicle version of the problem will be studied here. This section extends the result of Theorem 1 to three interconnected subsystems. Although the approach is similar, here the information available to the controllers shall be decomposed into three components instead of two, and hence the cost function will be decomposed accordingly. Since the scheme has been explained in detail in Section \ref{chapter2p}, a more concise derivation will be given here.

\subsection{Main Result}
\begin{theorem} \label{theorem2}
Assume that
\renewcommand{\theenumi}{\roman{enumi}}
\begin{small}\begin{enumerate}
                \item $(A,B)$, $(A{[2:3,2:3]},B
                {[2:3,2:3]})$, and $(A_{33},B_3)$ are stabilizable,
                \item $(Q,A),~(Q[2:3,2:3], ~A{[2:3,2:3]})$, and $(Q_{33},A_{33})$ are detectable.
\end{enumerate} \end{small}
Then, the optimal controller for the three-vehicle chain is given by:
\begin{align*}
\begin{bmatrix} \eta_1(t+1) \\ \eta_2(t+1)\end{bmatrix}&=(A-BL^1){[2:3,1:3]}\begin{bmatrix}x_1(t)\\\eta_1(t)\\\eta_2(t)\end{bmatrix} \\
\eta_3(t+1)&=(\tilde{A}-\tilde{B}L^2){[2,1:2]}\begin{bmatrix}x_2(t)-\eta_2(t)\\\eta_3(t)\end{bmatrix} \\
\begin{bmatrix} u_1(t) \\ u_2(t) \\ u_3(t)\end{bmatrix}&=
-L^1 \begin{bmatrix} x_1(t) \\ \eta_1(t) \\ \eta_2(t)\end{bmatrix}
-\begin{bmatrix} 0 \\ L^2\begin{bmatrix} x_2(t)-\eta_1(t) \\ \eta_3(t)\end{bmatrix} \end{bmatrix}  \nonumber \\
-&\begin{bmatrix}0 \\ 0 \\ L^3(x_3(t)-\eta_2(t)-\eta_3(t)) \end{bmatrix},
\end{align*}
and the \emph{optimal cost} is
$$\mathbf{Tr}(X^1_{11}W_1)+\mathbf{Tr}(X^2_{11}W_2)+\mathbf{Tr}(X^3W_3).$$
The matrices $X^1$, $X^2$, and $X^3$ are the positive semidefinite stabilizing solutions to the Riccati equations
\begin{align*} X^1=&A^TX^1A+Q-A^TX^1B(B^TX^1B+R)^{-1}B^TX^1A \\
               X^2=&\tilde{A}^TX^2\tilde{A}+\tilde{Q}-\tilde{A}^TX^2\tilde{B}(\tilde{B}^TX^2\tilde{B}+\tilde{R})^{-1}\tilde{B}^TX^2\tilde{A} \\
               X^3=&A_{33}^TX^3A_{33}+Q_{33} \\
                  -&A_{33}^TX^3B_3(B_3^TX^3B_3+R_{33})^{-1}B_{3}^TX^3A_{33}
\end{align*}
where $\tilde{A}=A{[2:3,2:3]}$, $\tilde{B}=B{[2:3,2:3]}$, $\tilde{Q}=Q{[2:3,2:3]}$ and $\tilde{R}=R{[2:3,2:3]}$.
The matrix $X^1$ is partitioned into blocks according to the partitions of $x$ as
$$X^1=[X^1_{ij}],~~i,j=1,..,3,$$
also, $X^2$ is partitioned according to the dimensions of $x_2$ and $x_3$ as
$$X^2=[X^2_{ij}],~~i,j=1,2.$$
The gain matrices are given by
\begin{align*}
L^1&=(R+B^TX^1B)^{-1}B^TX^1A, \\
L^2&=(\tilde{R}+\tilde{B}^TX^2\tilde{B})^{-1}\tilde{B}^TX^2\tilde{A}, \\
L^3&=(R_{33}+B_3^TX^3B_3)^{-1}B_3^TX^3A_{33}.
\end{align*}
\end{theorem}

\subsection{\textbf{Optimal Controller Derivation}}

The controllers in this case are restricted to the form
\begin{align*}
u(t)=\begin{bmatrix}f_{11}(x_1(0:t)) \\ f_{21}(x_1(0:t))+f_{22}(x_2(0:t)) \\f_{31}(x_1(0:t))+ f_{32}(x_2(0:t)) +f_{33}(x_3(0:t)) \end{bmatrix},
\end{align*}
where $f_{ij}$ are linear functions.

Again, the information shared among the controllers is $x_1(0:t)$, and hence each controller can run an estimator using this piece of information. Using this idea, we will decompose the state and control input into two independent terms. Let
\begin{align}
x(t)=z^1(t)+\tilde{z}(t),
\end{align}
where $z^1(t):=\mathbf{E}\{x(t)|x_1(0:t)\}$ and $\tilde{z}(t):=x(t)-z^1(t)$. Clearly, the first component of $z^1$ is $x_1$, and the other components of $z^1$ and $z^2$ will be labeled as:
$$z^1(t)=\begin{bmatrix}x_1(t)\\z^1_2(t)\\z^1_3(t) \end{bmatrix}, \tilde{z}(t)=\begin{bmatrix}0\\z^2_2(t)\\\tilde{z}_3(t) \end{bmatrix}. $$
The control variable will also be decomposed as $u(t)=u^1(t)+u^2(t)$, where $u^1(t):=\mathbf{E}\{u(t)|x_1(0:t)\}, \mbox{and } u^2(t):=u(t)-u^1(t)$ are independent terms. Similar to the argument made for two-vehicle problem, since $u_1(t)$ is function of the history of $x_1$, the first component of $u^1$ is $u_1$. As a result $u^1$ and $u^2$ will have the structure:
\begin{align}
u^1(t)=\begin{bmatrix}u_1(t)\\u^1_2(t)\\u^1_3(t) \end{bmatrix}, \tilde{u}(t)=\begin{bmatrix}0\\u^2_2(t)\\\tilde{u}_3(t) \end{bmatrix},
\end{align}
where \begin{equation} \label{constraints2} \begin{bmatrix}u^2_2(t)\\\tilde{u}_3(t)\end{bmatrix}=\begin{bmatrix}f_{22}(z^2_2(0:t)) \\ f_{32}(z^2_2(0:t))+f_{33}(\tilde{z}_3(0:t))\end{bmatrix}.\end{equation}

Similar to Lemma \ref{lemma2}, the following recursive equations can be found for $z^1$ and $\tilde{z}$:
\begin{align} \label{Z1}
z^1(t+1)&=Az^1(t)+Bu^1(t)+\begin{bmatrix}w_1(t)\\0\\0 \end{bmatrix}, \\
\tilde{z}(t+1)&=A\tilde{z}(t)+B\tilde{u}(t)+\begin{bmatrix}0\\w_2(t)\\w_3(t) \end{bmatrix}. \label{ztilde}
\end{align}
Since the first row in (\ref{ztilde}) is zero, this equation reduces to
\begin{align*}
\begin{bmatrix}
z^2_2(t+1) \\ \tilde{z}_3(t+1)
\end{bmatrix}&=\underbrace{\begin{bmatrix} A_{22}& 0\\ A_{32}&A_{33}\end{bmatrix}}_{\tilde{A}}\begin{bmatrix}
z^2_2(t) \\ \tilde{z}_3(t)
\end{bmatrix}\hspace{-1mm} \\
&+\hspace{-1mm}\underbrace{\begin{bmatrix} B_{2}& 0\\ 0 & B_{3}\end{bmatrix}}_{\tilde{B}}\begin{bmatrix} u^2_{2}\\ \tilde{u}_3\end{bmatrix}\hspace{-1mm}+\hspace{-1mm}\begin{bmatrix} w_2(t)\\w_3(t)\end{bmatrix}.
\end{align*}
This system with the information constraints stated in (\ref{constraints2}), has a similar structure to the two-vehicle problem. Hence, the states and inputs of (\ref{ztilde}) will be decomposed in a similar manner. In this case $z_2^2(0:t)$ is the information which is shared among the controllers, $u^2_{2}$ and $\tilde{u}_3$. Note that since controller 2 and controller 3 have access to $x_1$ and $x_2$, they both can construct $z^2_2(t)=x_2(t)-z_2^1(t)$ at each time step.

We get
\begin{align*}
\tilde{z}(t)&=\underbrace{\mathbf{E}\{\tilde{z}(t)|z_2^2(0:t)\}}_{z^2(t)}+\underbrace{\tilde{z}(t)-\mathbf{E}\{\tilde{z}(t)|z_2^2(0:t)\}}_{z^3(t)},\\
\tilde{u}(t)&=\underbrace{\mathbf{E}\{\tilde{u}(t)|z_2^2(0:t)\}}_{u^2(t)}+\underbrace{\tilde{u}(t)-\mathbf{E}\{\tilde{u}(t)|z_2^2(0:t)\}}_{u^3(t)},
\end{align*}
The components of $z^2(t)$, $z^3(t)$, $u^2(t)$, and $u^3(t)$ can be labeled as
\begin{align*}
z^2(t)\hspace{-1mm}&=\hspace{-1mm}\begin{bmatrix}0\\z^2_2(t)\\z^2_3(t) \end{bmatrix}\hspace{-1mm},z^3(t)\hspace{-1mm}=\hspace{-1mm}\begin{bmatrix}0\\0\\z^3_3(t)\end{bmatrix}, \\
u^2(t)\hspace{-1mm}&=\hspace{-1mm}\begin{bmatrix}0\\u^2_2(t)\\u^2_3(t) \end{bmatrix}\hspace{-1mm},u^3(t)\hspace{-1mm}=\hspace{-1mm}\begin{bmatrix}0\\0\\u^3_3(t) \end{bmatrix}
\end{align*}
where $z^2_3(t):=\mathbf{E}\{\tilde{z}_3(t)|z^2_2(0:t)\}$, $z^3_3(t):=\tilde{z}_3(t)-z^2_3(t)$, $u^2_3(t):=\mathbf{E}\{\tilde{u}_3(t)|u^2_2(0:t)\}$, and $u^3_3(t):=\tilde{u}_3(t)-u^2_3(t)$.
%where \begin{equation} \label{constraints3} \begin{bmatrix}u^2_2(t)\\u^2_3(t)\end{bmatrix}=\begin{bmatrix}f_{22}(z^2_2(0:t)) \\ f_{32}(z^2_2(0:t))+f_{33}(z^2_3(0:t))\end{bmatrix}\end{equation}.

Finally, we get the following orthogonal decomposition of the states and controls
\begin{align*}
x(t)&=\underbrace{\begin{bmatrix}x_1(t)\\z^1_2(t)\\z^1_3(t) \end{bmatrix}}_{z^1(t)}+\underbrace{\begin{bmatrix}0\\z^2_2(t)\\z^2_3(t) \end{bmatrix}}_{z^2(t)}+\underbrace{\begin{bmatrix}0\\0\\z^3_3(t)\end{bmatrix}}_{z^3(t)}, \\
u(t)&=\underbrace{\begin{bmatrix}u_1(t)\\u^1_2(t)\\u^1_3(t) \end{bmatrix}}_{u^1(t)}+\underbrace{\begin{bmatrix}0\\u^2_2(t)\\u^2_3(t) \end{bmatrix}}_{u^2(t)}+\underbrace{\begin{bmatrix}0\\0\\u^3_3(t) \end{bmatrix}}_{u^3(t)},
\end{align*}
where the dynamics of $z^1$ is given in (\ref{Z1}), and the dynamics for the non-zero component of $z^2$ and $z^3$ is given by
\begin{align} \label{Z22}
\begin{bmatrix}
z^2_2(t+1) \\ z^2_3(t+1)
\end{bmatrix}&=\tilde{A}\begin{bmatrix}
z^2_2(t) \\ {z}^2_3(t)
\end{bmatrix}\hspace{-1mm}+\hspace{-1mm}\tilde{B}\begin{bmatrix} u^2_{2}\\ {u}^2_3\end{bmatrix}\hspace{-1mm}+\hspace{-1mm}\begin{bmatrix} w_2(t)\\0\end{bmatrix},\\
z^3_3(t+1)&=A_{33}z_3^3(t)+B_{33}u^3_3(t)+w_3(t). \nonumber
\end{align}
Using Corollary \ref{cor1}, the cost function can be separated into three quadratic forms similar to the ones in (\ref{quadratic1}). Minimization of these quadratic forms gives the optimal controllers as
\begin{align} \label{controllers}
\begin{bmatrix} u_1^*(t) \\ u_2^*(t) \\ u_3^*(t)\end{bmatrix}=
-&L^1(t) \begin{bmatrix} x_1(t) \\ z_2^1(t) \\ z_3^1(t)\end{bmatrix}
-\begin{bmatrix} 0 \\ L^2(t)\begin{bmatrix} z_2^2(t) \\ z_3^2(t)\end{bmatrix} \end{bmatrix}  \nonumber \\
-&\begin{bmatrix}0 \\ 0 \\ L^3(t)z^3_3(t) \end{bmatrix},
\end{align}
where the gain matrices, $L^1$, $L^2$, and $L^3$ are given by
\begin{align*}
L^1(t)&=(R+B^TX^1(t)B)^{-1}B^TX^1(t)A, \\
L^2(t)&=(\tilde{R}+\tilde{B}^TX^2(t)\tilde{B})^{-1}\tilde{B}^TX^2(t)\tilde{A}, \\
L^3(t)&=(R_{33}+B_3^TX^3(t)B_3)^{-1}B_3^TX^3(t)A_{33},
\end{align*}
and $X^1$, $X^2$, and $X^3$ are the solutions to the Riccati equations
\begin{align*} X^1(t)=&A^TX^1(t+1)A+Q-A^TX^1(t+1)B\times    \\
                      &(B^TX^1(t+1)B+R)^{-1}B^TX^1(t+1)A, \\
               X^2(t)=&\tilde{A}^TX^2(t+1)\tilde{A}+\tilde{Q}-\tilde{A}^TX^2(t+1)\tilde{B}\times \\
                     & (\tilde{B}^TX^2(t+1)\tilde{B}+\tilde{R})^{-1}\tilde{B}^TX^2(t+1)\tilde{A}, \\
               X^3(t)=&A_{33}^TX^3(t+1)A_{33}+Q_{33}-A_{33}^TX^3(t+1)B_3\times  \\
               &(B_3^TX^3(t+1)B_3+R_{33})^{-1}B_{3}^TX^3(t+1)A_{33},
\end{align*}
where $\tilde{Q}=Q{[2:3,2:3]}$ and $\tilde{R}=R{[2:3,2:3]}$. $X^1(N)=Q$, $X^2(N)=Q{[2:3,2:3]}$, and $X^3(N)=Q_{33}$.

To find the mapping from $x$ to $u$, the update equations for the terms $z^1_2$, $z^1_3$, and $z^2_3$ must be obtained. By closing the loop in (\ref{Z1}) and (\ref{Z22}) by the corresponding controllers we obtain
\begin{align}
\begin{bmatrix} z^1_2(t+1) \\ z^1_3(t+1)\end{bmatrix}&=(A-BL^1){[2:3,1:3]}\begin{bmatrix}x_1(t)\\z^1_2(t)\\z^1_3(t)\end{bmatrix}, \\
z^2_3(t+1)&=(\tilde{A}-\tilde{B}L^2){[2,1:2]}\begin{bmatrix}x_2(t)-z^1_2(t)\\z^2_3(t)\end{bmatrix},
\end{align}
and finally the time-varying version of the controllers can be rewritten as
\begin{align} \label{controllers2}
\begin{bmatrix} u_1^*(t) \\ u_2^*(t) \\ u_3^*(t)\end{bmatrix}=
-&L^1(t) \begin{bmatrix} x_1(t) \\ z_2^1(t) \\ z_3^1(t)\end{bmatrix}
-\begin{bmatrix} 0 \\ L^2(t)\begin{bmatrix} x_2(t)-z_2^1(t) \\ z_3^2(t)\end{bmatrix} \end{bmatrix}  \nonumber \\
-&\begin{bmatrix}0 \\ 0 \\ L^3(t)(x_3(t)-z^1_3(t)-z^2_3(t)) \end{bmatrix}.
\end{align}
By letting $N$ go to infinity, the controllers converge to the stationary form given in Theorem 2. To compute the optimal cost in this case, let us partition the matrix $X^1$ into blocks $[X^1_{ij}],~i,j=1, 2, 3$ in accordance with the partitions of $A$, and do the same with $X^2$ according to the partitions of $\tilde{A}$ that is $X^2=[X^2_{ij}], i,j=1,2$.

Then, the infinite-horizon optimal cost is obtained by
\begin{align*}\lim_{N \rightarrow \infty} \frac{1}{N}\sum_{t=0}^{N-1}(&\mathbf{Tr}(X^1_{11}(t+1)W_1)+\mathbf{Tr}(X_{11}^2(t+1)W_2)) \\
+&\mathbf{Tr}(X^3(t+1)W_3)) \\
=&\mathbf{Tr}(X^1_{11}W_1)+\mathbf{Tr}(X_{11}^2W_2)+\mathbf{Tr}(X^3W_3).
\end{align*}

\section{\uppercase{numerical results}}
\label{sec:Simulations}

In this section, we implement the proposed controller on an $M=3$ HDV platoon (Figure~\ref{fig:platoon}) and evaluate the performance through a realistic scenario that HDV platoons often face on the road. We assume that the vehicles in the platoon can only measure the velocity and relative distance of the preceding vehicle and only receive information through wireless communication of all the preceding vehicles. This assumption is made to evaluate if a small addition in communication links could improve the system performance. Hence, a numerical comparison is made between the proposed controller and a suboptimal controller specifically designed for HDV platooning \cite{Alam11:Dec}.
The suboptimal controller uses local information, namely it only accounts for the dynamics of the preceding vehicle. Finally, we compare the proposed controller with the fully centralized linear quadratic controller.

When studying the behavior of vehicles within a finite platoon, the velocity does not deviate significantly from the lead vehicle's velocity trajectory. The control strategy is simply to provide an input that maintains the platoon velocity at a set relative distance. In practice, many random disturbances such as wind variation, changing topology, or varying road properties are inflicted upon the system. These disturbances are modeled as disturbances in state measurements. An additional disturbance of interest is a mandated deviation in the lead vehicles velocity. This often occurs due to varying traffic events that the lead vehicle must adhere to. Hence, integral action for the lead vehicle is also added as a state to the system presented in \eqref{eq:plantModel}, to model such disturbances. The system for a $M=3$ HDV platoon can thereby be grouped into sub-blocks, as in \eqref{eq:GeneralPlantModel}, for controller design, where

\begin{align*}
\begin{split}
A_{11}&=\begin{bmatrix} 0 & -1\\
0 & 1\end{bmatrix},~A_{ii}=\begin{bmatrix} 1 & -1\\
\delta_i & \theta_i\end{bmatrix},~A_{i(i-1)}=\begin{bmatrix} 0 & 1\\
0 & 0\end{bmatrix}\\
B_1&=\begin{bmatrix} 0\\
k_{u_1}\end{bmatrix},~B_2=\begin{bmatrix} 0\\
k_{u_2}\end{bmatrix},~B_3=\begin{bmatrix} 0\\
k_{u_3}\end{bmatrix},~i=2,3.
\end{split}
\end{align*}

The modeled HDVs are described as traveling in a longitudinal direction on a flat road. We consider a heterogeneous platoon, where the masses are set to $[m_1, m_2, m_3]=[30000, 40000, 30000]$\,kg. All the vehicles are assumed to be traveling in the steady state velocity  $v_0=19.44$\,m/s ($70$\,km/h) at time gap $\tau=1$\,s, which gives an intermediate distance of $d_0=19.44$. The maximum engine and braking torque for a commercial HDV varies based upon vehicle configuration but can be approximated to be 2500\,Nm and 60000\,Nm/Axle respectively.

State disturbances as well as several lead vehicle deviation disturbances are imposed on the system, see Figure~\ref{fig:disturbance}. The lead vehicle deviation disturbances can be explained by the following scenario. The platoon travels along a road where the road speed is 70\,km/h. Suddenly a slower vehicle enters the lane through a shoulder path (at the 45\,s time marker). The lead vehicle must therefore reduce its speed to 60\,km/h, in turn forcing the follower vehicles to reduce their speed and adapt their relative distance accordingly. After a while, the slower vehicle increases its speed to the road speed of 70\,km/h and no longer inhibits the platoon (120\,s time marker). Hence, the lead vehicle again resumes the road speed and the follower vehicles adapt the speed and distance automatically as well. Finally, the platoon arrives at a point where the road speed is changed to 80\,km/h (180\,s time marker).

Figure~\ref{fig:disturbance} shows the velocity trajectories of three HDV platoon in the top plot and the corresponding intermediate spacings in this scenario. The trajectories obtained through the optimal decentralized controller are bold. The trajectories are also plotted, with thinner lines, for the suboptimal decentralized controller. We see that the proposed optimal controller displays a good performance. The suboptimal controller displays a slightly harsher behavior with a faster speed change, since it does not take follower vehicles into account. Hence, the required relative control input energy is much higher for the suboptimal controller compared to our proposed controller, as can been seen in Figure~\ref{fig:inputTorque}. The first two rows in Table~\ref{tb:inputs} states the total control input energy required to handle the imposed disturbances. We can see that the optimal distributed controller reduces the control input energy by 10.4\,\% for the lead vehicle, by 16.3\,\% for the second vehicle, and by 15.5\,\% for the third vehicle. By estimating the states of the follower vehicles, the proposed controller mimics a centralized control strategy and displays a smoother behavior. However, the reduced control energy is obtained at the cost of adding a communication link between vehicle~1 and vehicle~3, since vehicle~1's state cannot be measured through the mounted radar on vehicle~3. Furthermore, the average velocity is reduced by $2.6$\,\%. Travel time is equally important for fleet operators. However, it is clear that there is a considerable saving in the fuel consumption at the cost of additional communication links and a much smaller reduction in travel time.

\begin{figure}[t]
\begin{center}
\includegraphics[width=7.4cm]{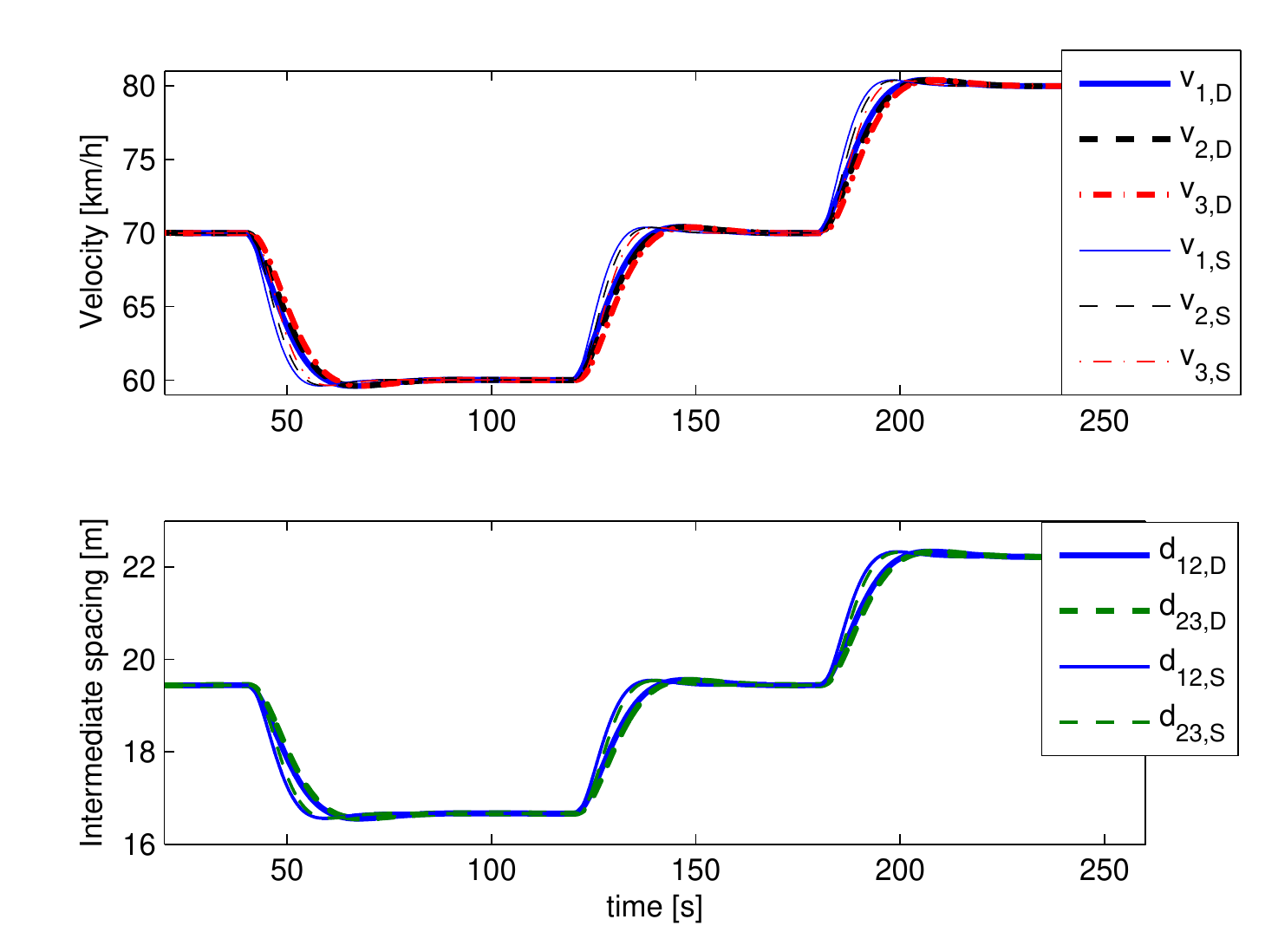}    % The printed column width is 8.4 cm.
\caption{Three HDV platoon, where a disturbance in velocity of the lead vehicle is imposed. The top plot shows the velocity trajectories for the $M=3$ HDV platoon and the bottom plot shows the intermediate spacings. The trajectories obtained through the optimal decentralized controller are bold and subindexed with $i,D$ and the trajectories obtained through the suboptimal controller are subindexed with $i,S$, where $i=1,2,3$ denote the platoon position index.}
\label{fig:disturbance}
\end{center}
\end{figure}

\begin{figure}[t]
\begin{center}
\includegraphics[width=7.4cm]{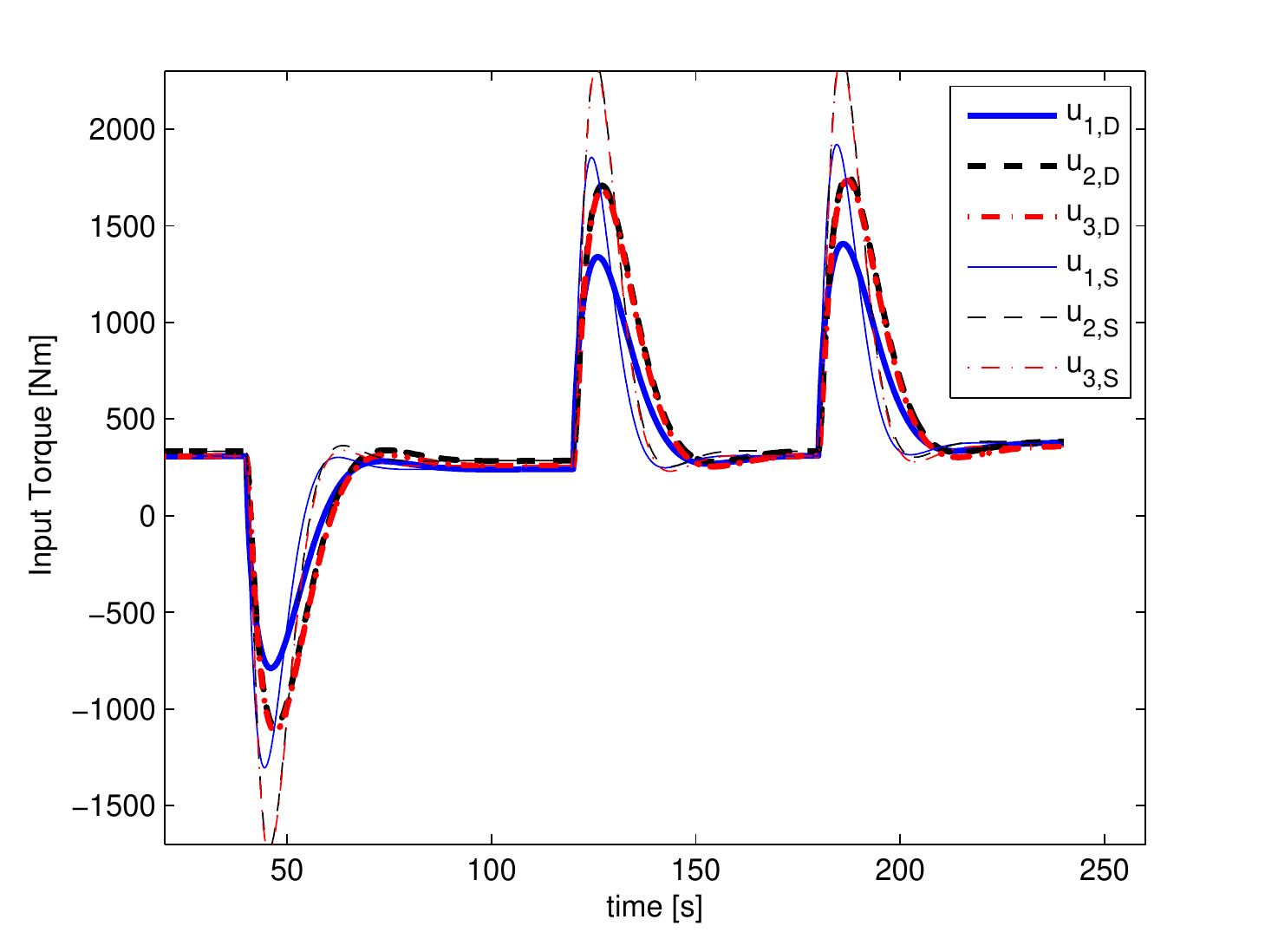}    % The printed column width is 8.4 cm.
\caption{Corresponding input torque to handle the imposed disturbances in Figure~\ref{fig:disturbance}. Similarly, the trajectories obtained through the optimal decentralized controller are bold and subindexed with $i,D$ and the trajectories obtained through the suboptimal controller are subindexed with $i,S$, $i=1,2,3$.}
\label{fig:inputTorque}
\end{center}
\end{figure}

Furthermore, the results in the last four rows of Table~\ref{tb:inputs} show that the required control input to handle the disturbances are well within the feasible physical range. The heaviest vehicle with platoon index $i=2$ naturally requires the largest control input to handle the disturbances. However, it is also seemingly where the highest relative difference in cost is obtained. By estimating the states of the follower vehicles, the mandated control input to handle the presented disturbances can be reduced significantly. Both the maximum and minimum values are lowered in the control input requirement for the optimal decentralized controller.

The optimal decentralized controller was also compared with a centralized control strategy. Since the proposed controller also accounts for all the states in the platoon by estimating the states of the follower vehicles, the behavior is close to the centralized controller. The computed relative differences in the cost function as well as the difference in required control inputs to handle the disturbances are minimal.
\begin{table}
\begin{center}
%\captionfont
\caption{Table of the required control input (Torque) to handle the disturbances in Figure~\ref{fig:inputTorque}.}
\label{tb:inputs}
\begin{tabular}{c|c|c|c}
$i$ & 1 & 2 & 3 \\\hline
$||u_{i,D}||_2$ [kNm] & 81.9 & 100.1 & 80.8 \\
$||u_{i,S}||_2$ [kNm] & 91.4 & 112.6 & 89.0 \\
$u_{i,D}^{max}$ [kNm] & 1.41 & 1.76 & 1.37 \\
$u_{i,S}^{max}$ [kNm] & 1.92 & 2.39 & 1.8 \\
$u_{i,D}^{min}$ [kNm] & -0.79 & -1.09 & -0.75 \\
$u_{i,S}^{min}$ [kNm] & -1.3 & -1.72 & -1.19 \\ \hline
\end{tabular}
\end{center}
\end{table}

\section{\uppercase{Conclusions}}
\label{sec:conclusion}
We have presented a quadratic optimal distributed control method for chain structures with applications to heterogeneous vehicle platooning under communication constraints. A procedure has been given for constructing low order optimal decentralized controllers through a simple decomposition scheme. A discrete HDV platoon model has been derived that includes physical coupling between the vehicles upon which the controllers are evaluated. The results show that the total control input energy required for the proposed controller is very close to a centralized controller where communication is needed among all the vehicles, and is significantly lower compared to a suboptimal controller which only accounts for the immediate preceding vehicle. In particular, by estimating the interaction with the follower vehicles, performance can be improved by adding a communication link from the first to the third vehicle in a three-vehicle platoon. Thus, considering preceding vehicles as well as follower vehicles is significant for fuel optimality.

A natural extension to the presented work is to derive explicit solutions for the problem of $M$-HDVs. Also, it would be interesting to consider time delays in the communication links between the vehicles. It is planned for future work.

%%%%%%%%%%%%%%%%%%%%%%%%%%%%%%%%%%%%%%%%%%%%%%%%%%%%%%%%%%%%%%%%%%%%%%%%%%%%%%%%
\bibliographystyle{plain}
\bibliography{References}

\begin{thebibliography}{10}

\bibitem{Alam11:Dec}
Assad Alam, Ather Gattami, and Karl~Henrik Johansson.
\newblock Suboptimal decentralized controller design for chain structures:
  Applications to vehicle formations.
\newblock In {\em 50th IEEE Conference on Decision and Control and European
  Control Conference}, Orlando, FL, USA, December 2011.

\bibitem{astrom-wittenmark}
Karl~J {\AA}strom and B.~Wittenmark.
\newblock {\em Computer Controlled Systems: Theory and Design}.
\newblock Prentice-Hall, 1984.

\bibitem{BamiehJovanovic05}
Bassam Bamieh and Mihailo~R. Jovanovi\'{c}.
\newblock On the ill-posedness of certain vehicular platoon control problem.
\newblock {\em IEEE Transactions on Automatic Control}, 50(9), September 2005.

\bibitem{Bamieh08}
Bassam Bamieh, Mihailo~R. Jovanovi\'{c}, Partha Mitra, and Stacy Patterson.
\newblock Effect of topological dimension on rigidity of vehicle formations:
  Fundamental limitations of local feedback.
\newblock In {\em 47th IEEE Conference on Decision and Control}, pages 369
  --374, Cancun, Mexico, Dec. 2008.

\bibitem{bamie}
Bassam Bamieh and Petros~G. Voulgaris.
\newblock A convex characterization of distributed control problems in
  spatially invariant systems with communication constraints.
\newblock {\em Systems and Control Letters}, 54(6):575 -- 583, 2005.

\bibitem{Barooah05}
Prabir Barooah and Jo$\tilde{\textrm{a}}$o~P. Hespanha.
\newblock Error amplification and disturbance propagation in vehicle strings
  with decentralized linear control.
\newblock In {\em 44th IEEE Conference on Decision and Control and the European
  Control Conference}, pages 1350 -- 1354, Seville, Spain, December 2005.

\bibitem{DaimlerChrysler}
Christophe Bonnet and Hans Fritz.
\newblock Fuel consumption reduction in a platoon: Experimental results with
  two electronically coupled trucks at close spacing.
\newblock In {\em Future Transportation Technology Conference \& Exposition},
  Costa Mesa, CA, USA, August 2000.
\newblock SAE paper 2000 - 01 - 3056.

\bibitem{DeSchutter:99}
B.~{De Schutter}, T.~Bellemans, S.~Logghe, J.~Stada, B.~{De Moor}, and
  B.~Immers.
\newblock Advanced traffic control on highways.
\newblock {\em Journal A}, 40(4):42--51, December 1999.

\bibitem{thesis}
Ather Gattami.
\newblock {\em Optimal Decisions with Limited Information}.
\newblock Department of Automatic Control, Lund University, 2007.

\bibitem{c5}
Yu-Chi Ho and K'ai-Ching Chu.
\newblock Team decision theory and information structures in optimal control
  problems--part 1.
\newblock {\em IEEE Transactions on Automatic Control}, 17(1):15 -- 22, Feb
  1972.

\bibitem{IoannouChien93}
P.A. Ioannou and C.C. Chien.
\newblock Autonomous intelligent cruise control.
\newblock {\em IEEE Transactions on Vehicular Technology}, 42(4):657 --672,
  Nov. 1993.

\bibitem{rantzer2006}
A.~Rantzer.
\newblock Linear quadratic team theory revisited.
\newblock In {\em American Control Conference}, June 2006.

\bibitem{RoggeAeyels}
Jonathan Rogge and Dirk Aeyels.
\newblock Decentralized control of vehicle platoons with interconnection
  possessing ring topology.
\newblock {\em 44th IEEE Conference on Decision and Control, and the European
  Control Conference}, pages 1491--1496, 2005.

\bibitem{Sahlholm11}
Per Sahlholm.
\newblock {\em Distributed Road Grade Estimation for Heavy Duty Vehicles}.
\newblock PhD thesis, Royal Institute of Technology (KTH), 2011.

\bibitem{num1}
C.~W. Scherer.
\newblock Structured finite-dimensional controller design by convex
  optimization.
\newblock {\em Linear Algebra and its Applications}, 351–352:639 -- 669, 2002.

\bibitem{shah10}
P.~Shah and P.~A. Parrilo.
\newblock $\mathcal{H}_2$-optimal decentralized control over posets: A state
  space solution for state-feedback.
\newblock In {\em Decision and Control (CDC), 2010 49th IEEE Conference on},
  pages 6722 --6727, December 2010.

\bibitem{SudinCook04}
Shahdan Sudin and Peter~A. Cook.
\newblock Two-vehicle look-ahead convoy control systems.
\newblock {\em 59th IEEE Vehicular Technology Conference, VTC}, 5:2935 -- 2939,
  2004.

\bibitem{Hedrick96}
D~Swaroop and JK~Hedrick.
\newblock String stability of interconnected systems.
\newblock {\em IEEE Transactions on Automatic Control}, 41(3):349 --357, March
  1996.

\bibitem{swigart10}
J.~Swigart and S.~Lall.
\newblock Optimal synthesis and explicit state-space solution for a
  decentralized two-player linear-quadratic regulator.
\newblock In {\em Decision and Control (CDC), 2010 49th IEEE Conference on},
  pages 132 --137, December 2010.

\bibitem{Varaiya93}
Pravin Varaiya.
\newblock Smart cars on smart roads: Problem of control.
\newblock {\em IEEE Transactions on Automatic Control}, 38(2), February 1993.

\bibitem{c9}
P.~G. Voulgaris.
\newblock Control of nested systems.
\newblock In {\em American Control Conference, 2000. Proceedings of the 2000},
  volume~6, pages 4442 --4445, 2000.

\bibitem{c4}
H.~S. Witsenhausen.
\newblock Separation of estimation and control for discrete time systems.
\newblock {\em Proceedings of the IEEE}, 59(11):1557 -- 1566, Nov. 1971.

\bibitem{Hucho}
Hucho Wolf-Heinrich and Syed~R. Ahmed.
\newblock {\em Aerodynamics of Road Vehicles}.
\newblock Society of Automotive Engineers, Inc, Warrendale, 1998.

\bibitem{num2}
D.~Zelazo and M.~Mesbahi.
\newblock H2 analysis and synthesis of networked dynamic systems.
\newblock In {\em American Control Conference, 2009. ACC '09.}, pages 2966
  --2971, June 2009.

\end{thebibliography}
\

\section{\uppercase{Appendix}}

\emph{Proof of Lemma 2}:

We have $z^1(t+1)=\mathbf{E}\{x(t+1)|x_1(0:t+1)\}=\mathbf{E}\{x(t+1)|x_1(0:t),x_1(t+1)\}$.
To evaluate the conditional expectation given $x_1(0:t)$ and $x_1(t+1)$, we change the variables so that we get independent variables. Note that we can construct $x_1(t+1)=A_{11}x_1(t)+B_1u_1(t)+w(t)$ given that $x_1(0:t)$ and $w_1(t)$ are available. Hence, instead of evaluating the conditional expectation of $x(t+1)$ given $x_1(0:t)$ and $x_1(t+1)$, we will evaluate the expectation given the independent variables $x_1(0:t)$ and $w_1(t)$. In other words, $w_1(t)$ is part of $x_1(t+1)$ which was not previously available in $x_1(0:t)$. Thus we find that,
\begin{align*}
z^1(t+1)=&\mathbf{E}\{Ax(t)+Bu(t)+w(t)|x_1(0:t), w_1(t)\} \nonumber  \\
        =&\mathbf{E}\{Ax(t)+Bu(t)+w(t)|x_1(0:t)\} \nonumber  \\
        +&\mathbf{E}\{Ax(t)+Bu(t)+w(t)|w_1(t)\}  \nonumber \\
        =&Az^1(t)+Bu^1(t)+\begin{bmatrix}w_1(t)\\ 0 \end{bmatrix},\end{align*}
where we also used the pairwise independence of $w_1(t)$, $w_2(t)$, $x(t)$, and $u(t)$.
So, $z^2$ is given by
\begin{align*}  z^2(t+1)&=x(t+1)-\mathbf{E}\{x(t+1)|x_1(0:t+1)\} \nonumber \\
                       &=Az^2(t)+Bu^2(t)+\begin{bmatrix}0\\ w_2(t) \end{bmatrix}.\end{align*}

\end{document}